\documentclass[a4paper, 12pt]{amsart}

\usepackage{latexsym,amscd,amssymb,amsmath,amsthm,comment,url,color,mathrsfs}
\usepackage{graphicx,xcolor}
\usepackage{pb-diagram,tikz}
\usetikzlibrary{positioning,arrows.meta,decorations.pathreplacing,decorations.pathmorphing}
\usepackage[all]{xy}

\title[Local cohomologies at ideals given by simplicial posets]{Toward the theory on local cohomologies at the ideals given by simplicial posets}

\author{Kosuke Shibata}
\address{National Institute of Technology, Yonago College, Yonago,Tottori, 683-8502, JAPAN}
\email{shibata@yonago.kosen-ac.jp}

\author{Kohji Yanagawa} 
\address{Department of Mathematics, 
Kansai University, Osaka, 564-8680, Japan.}
\email{yanagawa@kansai-u.ac.jp}

\dedicatory{Dedicated to the memory of Professor Mitsuyasu Hashimoto.}


\theoremstyle{plain}
  \newtheorem{thm}{Theorem}[section]
  \newtheorem{prop}[thm]{Proposition}
  \newtheorem{lem}[thm]{Lemma}

\theoremstyle{definition}
  \newtheorem{dfn}[thm]{Definition}
  \newtheorem{exmp}[thm]{Example}
\theoremstyle{remark}
  \newtheorem{rem}[thm]{Remark}
\let\opn\operatorname 

\def\@bothmode#1{\ifmmode #1\else $#1$\fi}
\def\@autopr#1{%
   \@chCount{#1}%
   \ifnum\@tempchn<2 #1\else (#1)\fi
}
\numberwithin{equation}{section}
\def\NN{\mathbb{N}} 
\def\ZZ{\mathbb{Z}} 

\def\m{\ideal{m}} 
\def\p{\ideal{p}} 
\def\bc{a} 
\def\chara{\operatorname{char}} 
\def\defopn#1{%
    \@xp\def\csname #1\endcsname{%
        \def\@tempopn{\opn{\csname the#1\endcsname}}%
        \@ifnextchar_{\@opform}{\@opform_{}}%
    }%
}
\let\ideal\mathfrak 
\def\theE{E}
\def\E{\@ifstar{{}^*\theE}{\theE}} 
\def\*Mod{{}^*\!\Mod}
\def\Mod{\operatorname{Mod}}

\def\<{\langle}
\def\>{\rangle}
\def\rank{\operatorname{rank}}
\def\hght{\operatorname{ht}}

\def\chara{\operatorname{char}}
\def\<{\langle}
\def\>{\rangle}

\def\DC{D^\bullet}

\def\fm{\mathfrak{m}}
\def\fp{\mathfrak{p}}
\def\fq{\mathfrak{q}}

\def\ba{\mathbf a}
\def\bb{\mathbf b}
\def\bc{\mathbf c}
\def\bd{\mathbf d}
\def\be{\mathbf e}
\def\b0{\mathbf 0}
\def\bt{\mathbf t}
\def\bw{\mathbf w}

\def\11{\mathbf 1}

\def\EE{{}^*\! E}

\def\wS{\widetilde{S}}

\def\wE{{}^*\!\widetilde{E}}

\def\EA{{}^*\! E_{S_x}(S_x)}
\def\wEA{{}^*\! E_{\wS_x}(\wS_x)}

\def\ESx{{}^*\! E_{S_{-x}}(K)}
\def\wESx{{}^*\! E_{\wS_{-x}}(K)}

\def\too{\longrightarrow}
\def\wt{\widetilde{t}}
\newcommand{\disp}{\displaystyle}
\usepackage{bm}

\def\Ass{\operatorname{Ass}}
\def\Ext{\operatorname{Ext}}
\def\Hom{\operatorname{Hom}}
\def\gHom{^*\!\Hom}
\def\gExt{^*\!\Ext}

\def\Spec{\operatorname{Spec}}

\def\wt{{\widetilde{t}}}

\def\wS{{\widetilde{S}}}

\def\fp{{\mathfrak{p}}}
\def\bfp{\overline{{\mathfrak{p}}}}

\def\Lf{\mathcal{L}}

\def\DC{D^\bullet}
\def\gD{{}^*\!D}
\def\DB{{}^*\! \DC}

\def\bt{\mathbf t}
\def\too{\longrightarrow}
\def\gE{{}^*\!E}
\def\gD{{}^*\! D}

\def\Im{\operatorname{Im}}
\def\Ker{\operatorname{Ker}}

\makeatother
\begin{document}
\begin{abstract}
For a {\it simplicial poset} $P$, Stanley assigned the face ring $A_P$, which is the quotient of the polynomial ring $S:=K[t_x \mid x \in P \setminus \{\widehat{0} \}]$ by the ideal $I_P$. This is a generalization of Stanley-Reisner rings, but $S$ and $A_P$ are not standard graded in this case, and $I_P$ is not  a monomial ideal.  
To establish the foundation of the theory on local cohomology $H_{I_p}^i(S)$ and its injective resolution, we give an explicit description of the graded injective envelope $\EE_S(S/\fp_x)$, where $\fp_x$ is the prime ideal associated with $x \in P$, and analyze their behavior in the graded dualizing complex. 
\end{abstract}

\maketitle

\section{Introduction}
Throughout the paper, let $P$ be a {\it simplicial poset}, that is, $P$ is a finite poset with the smallest element $\hat{0}$, and the subposet $\{ y \in P \mid y \leq x  \}$ is isomorphic to a boolean algebra for all $x \in P$. 
If we regard the set of faces of a simplicial complex as a poset by inclusions, it is simplicial. Moreover, any simplicial poset is isomorphic to the face poset of a particular class of regular cell complexes. 
Moreover, simplicial posets and their face rings play roles in toric topology (c.f. \cite{M05,MP06}). 

For $x_1, \ldots, x_m \in P$, $[x_1 \vee \cdots \vee x_m ]$ denotes the set of minimal elements of $\{ \, z \in P \mid z \geq x_1, \ldots, x_m \, \}$. 
If $[ x \vee y] \ne \emptyset$, then  
$\{ \, z \in P \mid z \leq x, y \, \}$ has the greatest element $x \wedge y$. 
Let us recall the construction  of Stanley \cite{St91}. 
For $P^*:=P\backslash \{\hat{0}\}$,   
let  $S := K[ \, t_x \mid x \in P^* \, ]$ be the polynomial ring over a field $K$. Set 
\begin{equation}\label{f_xy}
f_{x,y}:=t_xt_y - t_{x \wedge y} \sum_{z \in [x \vee y]} t_z
\end{equation}
for $x,y \in P^*$. Here we set  $t_{\hat{0}}=1$ for the convenience, and 
if $[x \vee y] = \emptyset$, we interpret that 
$f_{x,y} =t_xt_y$.     
Now consider the ideal  
$$I_P:= (\, f_{x,y} | \ x,y \in P^*   \, ) \subset S$$
and call $A_P:= S/I_P$ the {\it face ring} of $P$.  See Example~\ref{running} below for a concrete illustration. If the simplicial poset $P$ comes from a simplicial complex $\Delta$, then $A_P$ is isomorphic to the Stanley--Reisner ring of $\Delta$. 

Let $y_1, \ldots, y_n$ be the rank 1 elements of $P$. 
Then $S$ admits a $\ZZ^n$-grading with $\deg t_{y_i} =\be_i \in \NN^n$, where $\be_i$ denotes the $i$-th coordinate vector. Then $I_P$ is a $\ZZ^n$-graded ideal, and $A_P$ is a $\ZZ^n$-graded ring. 
This grading induces the $\ZZ$-grading of $S$, but then it is {\it not} standard (i.e., some variable has degree $\ge 2$). 
Moreover, we often have $\dim_K S_\ba \ge 2$ for $\ba \in \NN^n$. 
These features  make our grading far more involved
than the usual $\ZZ^n$-grading of $K[x_1, \ldots, x_n]$.

Let $(R, \fm)$ be a regular local ring of dimension $d$ containing a field, $I$ an ideal of $R$, 
and $A:=R/I$ the quotient ring. The local cohomology module $H_I^i(R)$ is very mysterious, but  the {\it Lyubeznik number} 
$$\lambda_{i,j}(A):= \mu_i(\fm,  H_I^{d-j}(R))$$
is always finite and depends only on the ring $A$ (\cite{L93,L97}). 
For a polynomial ring $T$ and a graded ideal $I$, $\lambda_{i,j}(T/I)$ is defined in the same way. 
While many excellent papers have been written on this notion, refined and detailed research has been done for the case of Stanley–Reisner ideals (i.e.,  squarefree monomial ideals). See \cite{Y01a,AV}.

In the present paper, we establish the foundation for the study on $H_{I_P}^i(S)$.  
Every $\ZZ^n$-graded prime ideal of $S$ containing $I_P$ is of the form $\fp_x:=(t_z \mid z \in P^* \, \text{with} \, z \not \le x)$ for some $x \in P$. Let $\EE_x$ be the $\ZZ^n$-graded injective envelope $\EE_S(S/\fp_x)$ of $S/\fp_x$. 
For the $\ZZ^n$-graded dualizing complex ${}^*\! D_S^\bullet$ of $S$, the complex $\Gamma^\bullet:=
\Gamma_{I_P}({}^*\! D_S^\bullet)$ is of the form 
$$\Gamma^{-i}=\bigoplus_{\substack{x \in P \\\rank x=i}} \EE_x,$$
and satisfies 
$$H^i(\Gamma^\bullet)\cong H_{I_P}^{i+N}(S)(-\bw),$$
where $N:=\dim S$ and $\bw:=\sum_{x \in P^*} \deg t_x \in \NN^n$. 

To develop the theory, we face the following difficulties. 
First, the structure of $\EE_x$ is considerably more complicated than in the case of Stanley–Reisner ideals.
Second, the set of $\mathbb{Z}^n$-graded $S$-homomorphisms $\EE_x \to \EE_{x'}$ is much harder to control than the classical case.

The organization of this paper is as follows.

\begin{itemize}
\item In \S 2, we give an explicit description of $\EE_x$ for $x \in P$. For example, if $x \in [y_1, \ldots, y_r]$, we have 
$$\EE_x=K[t_{y_1}^{\pm 1}, \ldots, t_{y_r}^{\pm 1}] \otimes_K 
K[t_z^{-1} \mid z \in P^* \setminus \{ y_1, \ldots, y_r\}].$$
The $S$-module structure is given via the ``comultiplication map" $S \ni t_z \longmapsto t_z\otimes 1 +1 \otimes t_z \in S \otimes_K S$ and the natural surjection $S \to S/\fp_x (\cong K[y_1, \ldots, y_r])$. 

\item In \S3, we define the {\it cleanness} for homomorphisms $\EE_x \to \EE_{x'}$. The cleanness depends on the choice of the $K$-bases (i.e., the sets of monomials) of $\EE_x$ and $\EE_{x'}$. If we fix the bases, a clean homomorhism is unique up to a  scalar multiplication.   
\item We show that the differential maps of $\Gamma^\bullet$ are clean, after suitable choice of bases (Theorem~\ref{diff are clean}).  
\end{itemize}

As an application of the present framework, in the forthcoming paper \cite{SY2}, we will  establish the formula 
$$\lambda_{i,j}(A_P)=\dim_K [\gExt_S^{n-i}(\gExt_S^{n-j}(A_P,S),S)]_{\mathbf{0}},$$
which generalizes the key result  \cite[Corollary 3.10]{Y01a} in the case of Stanley–Reisner ideals. We remark that, since $A_P$ is not standard graded, we cannot use \cite[Theorem~4.5]{DGN},  even if $\chara(K)>0$ (then $A_P$ is $F$-pure). Moreover, our method is  constructive. 
For example, we have a presentation 
$$0 \too H_{I_P}^i(S)(-\bw) \too \bigoplus_{x \in P} \EE_x^{l_x} \stackrel{\psi}{\too}  \bigoplus_{x \in P} \EE_x^{m_x},$$
where $\psi$ is clean. However, the proof requires totally different methods (e.g., a term order in the Laurent polynomial ring $S[t_{y_1}^{-1}, \ldots, t_{y_n}^{-1}]$), and we will  present it in a separate paper \cite{SY2}.


\section{Construction of multigraded injective envelope}
Let $P$ be a {\it simplicial poset}. 
For the definition of simplicial posets and related notions (especially the face ring $A_P$), see \S 1. 

\begin{exmp}\label{running}
The poset $P_1$ given by the following Hasse diagram is simplicial. We have $[y_1 \vee y_2]=\{x,z\}$, $y_1 \wedge y_2= \hat{0}$ and $[x \vee z]=\emptyset$ ($x \wedge z$ does not exist). 
The polynomials defined in \eqref{f_xy} are 
$f_{y_1, y_2}=t_{y_1}t_{y_2}-(t_x+t_z)$ and $f_{x,z}=t_xt_z$, and the face ring $A_P$ is given by 
$$S/I_{P_1}=K[t_{y_1}, t_{y_2}, t_x, t_y]/(t_{y_1}t_{y_2}-t_x-t_z, t_xt_z)$$
(note that, for a general $P$, if $x$ and $y$ are comparable, then $f_{x,y}=0$.)

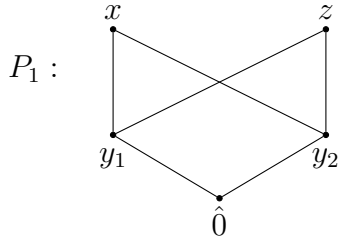
\begin{center}
\begin{figure}[ht]
\begin{tikzpicture}[scale=0.7, every node/.style={circle, draw, fill=black, inner sep=0.7pt}]

\node[label=above:$P_1:$,color=white] (a) at (-3.5,0.6) {};
\node[label=above:$x$] (x) at (-2,2) {};
\node[label=above:$z$] (z) at (2,2) {};
\node[label=below:$y_1$] (y_1) at (-2,0) {};
\node[label=below:$y_2$] (y_2) at (2,0) {};
\node[label=below:$\hat{0}$] (0) at (0,-1.2) {};

\draw (x) -- (y_1);
\draw (x) -- (y_2);
\draw (z) -- (y_1);
\draw (z) -- (y_2);
\draw (y_1) -- (0);
\draw (y_2) -- (0);
\end{tikzpicture}
\end{figure}
\end{center}

\end{exmp}

Let $y_1, \ldots, y_n$ be the rank 1 elements of $P$. For simplicity, set $t_i:=t_{y_i}$. For $x \in P$ with $\rank x =1$, $t_x$ coincides with some $t_i$, but in some contexts we write $t_x$ as it is. 
We regard $S=K[t_x \mid x \in P^*]$ as a $\ZZ^n$-graded ring by $\deg x =\sum_{y_i \le x} \be_i \in \NN^n$ for each $x \in P^*$, where $\be_i \in \NN^n$ is the $i$-th coordinate vector.  Then $I_P$ is a $\ZZ^n$-graded ideal, and hence $A_P$ is a $\ZZ^n$-graded ring. 

There are infinitely many $\ZZ^n$-graded prime ideals of $S$ in general. For example, if $x \in[y_i \vee y_j]$ for $i \ne j$, then $(t_it_j-c x) \subset S$ is a $\ZZ^n$-graded prime ideal for $0 \ne c \in K$, and these are all distinct. However, the $\ZZ^n$-graded prime ideals of $S$ containing $I_P$ correspond to the $\ZZ^n$-graded prime ideals of $A_P$, which are discussed in \cite{Y11}. So these prime ideals are of the form 
$$\fp_x:=(t_z \mid z \not \le x)+I_P$$ for some $x \in P$. Note that $\fp_{\hat{0}}$ is the graded maximal ideal $\fm:=(t_z \mid z \in P^*)$. 
Set $$S_x:=S/\fp_x \cong K[t_i \mid y_i \le x]$$
(note that $S_{\hat{0}} =S/\fm \cong K$). 
For simplicity, we denote the image of $t_i$ for $y_i \le x$ in $S_x$ just by $t_i$, not something like $\overline{t_i}$. (The same is true for the other quotient ring $S_{-x}$ defined below.)

\begin{exmp}
Let $P_1$ be the simplicial poset given in  Example~\ref{running}. 
The $\ZZ^2$-grading of $S=K[t_1, t_2, t_x,t_z]$ is given by $\deg t_1=(1,0)$, $\deg t_2=(0,1)$, $\deg t_x =\deg t_z=(1,1)$, and the $\ZZ^2$-graded prime ideals containing $I_P$ are 
$\fp_x=(y_1y_2-x, z), \quad \fp_z=(y_1y_2-z, x)$,  $\fp_{y_1}=(y_2, x,z)$, $\fp_{y_2}=(y_1, x,z)$,  $\fp_{\hat{0}}=(y_1,y_2,x,z)$. 
The surjection $\pi_x: S \to S_x \cong K[y_1, y_2]$ is given by  $\pi_x(x)=t_1t_2$ and $\pi_x(z)=0$. Similarly, $\pi_z: S \to S_z \cong K[y_1, y_2]$ is given by  $\pi_x(x)=0$ and $\pi_x(z)=t_1t_2$. 
\end{exmp}

The classical paper \cite{GW} on $\ZZ^n$-graded rings and modules has been read as a reference for monomial ideals or affine semigroup rings in many cases. However, the results in Chapters 1 and 2 of this paper also cover our somewhat unusual case. The same is true for Chapters 13 and 14 of the relatively recent textbook \cite{BS}. 
The next paragraph is based on the beginning of p.243 of \cite{GW} and p. 256 of \cite{BS}. 

Let $\*Mod S$ be the category of $\ZZ^n$-graded $S$-modules and degree preserving $S$-homomorphisms between them. For $M \in \*Mod S$ and $\ba \in \ZZ^n$, 
$M(\ba)$ denotes the $\ZZ^n$-graded $S$-module which coincides with $M$ as the underlying $S$-module
and the grading is given by $M(\ba)_{\bb}=M_{\ba +\bb}$ for $\bb \in \ZZ^n$. For $M, N \in \*Mod S$ and $\ba \in \ZZ^n$, $[\gHom_S(M,N)]_\ba$ denotes the set of morphisms $M \to N(\ba)$ in $\*Mod S$. (In particular, $[\gHom_S(M,N)]_\b0$ is the set of morphisms $M \to N$ in $\*Mod S$.) Then
$$\gHom_S(M,N)=\bigoplus_{\ba \in \ZZ^n} [\gHom_S(M,N)]_\ba$$
has a natural $\ZZ^n$-graded $S$-module structure.
If $M$ is finitely generated, then the underlying module  of $\gHom_S(M,N)$ is isomorphic to $\Hom_S(M,N)$. 
The category $\*Mod S$ has enough projectives, and an indecomposable projective is isomorphic to $S(\ba)$ for some $\ba \in \ZZ^n$.  For $M, N \in \*Mod S$, we can define $[\gExt^i_S(M,N)]_\ba$ for $\ba\in \ZZ^n$, and  
$$\gExt^i_S(M,N)=\bigoplus_{\ba \in \ZZ^n} [\gExt^i_S(M,N)]_\ba \in \*Mod S$$ as before. Since $S$ is noetherian, if $M$ is finitely generated, then the underlying module  of $\gExt^i_S(M,N)$ is isomorphic to $\Ext^i_S(M,N)$. 

 For a $\ZZ^n$-graded prime ideal $\fp \subset S$, 
$\EE_S(S/\fp)$  denotes the injective envelope of $S/\fp$ in $\*Mod S$, or {\it *injective envelope} of $S/\fp$. 
 To clarify the terminology, we state the properties of $\EE_S(S/\fp)$ here. 
\begin{itemize}
\item[(1)] $\EE_S(S/\fp)$ contains $S/\fp$ 
as a submodule, and a {\it *essential extension} of $S/\fp$.  The latter condition means that a non-zero homogeneous element  $\alpha \in \EE_S(S/\fp)$ always satisfies $S \alpha \cap (S/\fp_x)\ne 0$.  
\item[(2)] $\EE_S(S/\fp)$ is injective in $\*Mod S$. 
\end{itemize}
For a general $\ZZ^n$-graded prime ideal $\fp$, it is hopeless  to give an ``explicit" description of $\EE_S(S/\fp)$. However, in this section, we will explicitly construct $\EE_S(S/\fp_x)$ for $x \in P$.

Recall that, for a general polynomial ring $S'=K[x_1, \ldots, x_m]$, the injective envelope $\EE_{S'}(K)$ of $K=S'/(x_1, \ldots, x_m)$ is isomorphic to Macaulay's inverse system $K[x_1^{-1}, \ldots, x_m^{-1}]$ regardless of the grading. Hence  $\EE_{\widehat{0}}=\EE_S(K)$ is isomorphic to $K[t_z^{-1} \mid z \in P^*]$. 
So we consider $x \in P^*$ with $r:= \rank_P x \ge 1$. We may assume that $\{ i \mid y_i \le x \}=\{1, \ldots, r \}$ (equivalently, $x\in [y_1 \vee \cdots \vee y_r]$). Note that $S_x\cong K[t_1, \ldots, t_r]$ now. 

The injective envelope $\EA$ of $S_x$ in  the category $\ZZ^r$-graded $S_x$-modules is isomorphic to the Laurent polynomial ring  $K[t_1^{\pm 1}, \ldots, t_r^{\pm 1}]$, which is a localization of (the $S_x$-module) $S_x$. 

Next, we set $P^*_{-x}:=P^* \setminus \{ y_1, \ldots, y_r \}$, and consider the quotient ring 
$$S_{-x} := S/(t_1, \ldots, t_r) \cong K[t_z \mid   z\in P^* _{-x}],$$
which inherits the $\ZZ^n$-grading from $S$. 
The injective envelope $\ESx$ of $K$  in  the category $\ZZ^n$-graded $S_{-x}$-modules is 
$K[t_z^{-1} \mid z \in P^*_{-x}]$. 

We can regard $\EA$ and $\ESx$ as $S$-modules through the canonical surjections $S \to S_x \, (=S/\fp_x)$ and $S \to S_{-x} \, (=S/(t_1, \ldots, t_r))$. 
As a candidate of the injective envelope $\EE_S(S_x)$ of $S_x \, (= S/\fp_x)$ in $\*Mod S$, we consider the  $S \otimes_k S$-module 
$$
\EE_x := \EA \otimes_K \ESx.  
$$
The set of elements of the form 
\begin{equation}\label{monomials in E_x} 
    \alpha = \prod_{i=1}^r t_i^{a_i} \otimes \prod_{z \in P^*_{-x}} t_z^{-a_z}
\end{equation}
for $(a_1, \ldots, a_r) \in \ZZ^r$ and $(a_z) \in \NN^{P^*_{-x}}$  
forms a $K$-basis of $\EE_x$. We call an element of this form a {\it monomial} of $\EE_x$. 
Through the ring homomorphism $\Delta: S \to  S \otimes_k S$ defined
by $$\Delta(t_z) = t_z \otimes 1 + 1 \otimes t_z$$ for all $z \in P^*$, we regard $\EE_x$ as an $S$-module. 
We can also regard $\EE_x$ as a (free) module over the Laurent polynomial ring $K[t_1^{\pm 1}, \ldots, t_r^{\pm 1}]$ by 
\begin{equation}\label{E_x as a module over Laurent}
\prod_{i=1}^r t_i^{b_i} \cdot \alpha = \prod_{i=1}^r t_i^{a_i+b_i} \otimes \prod_{z \in P^*_{-x}} t_z^{-a_z}
\end{equation}
for $\prod_{i=1}^r t_i^{b_i} \in K[t_1^{\pm 1}, \ldots, t_r^{\pm 1}]$ and the monomial $\alpha \in \EE_x$ in \eqref{monomials in E_x}. 

The $\ZZ^n$-grading of $\EE_x$ is given as follows. For a homogeneous element $u \in \EA$ (resp. $v \in \ESx$)
with $\deg_{\ZZ^r}(u)=(a_1, \ldots, a_r)$ (resp. $\deg_{\ZZ^n}(v)=(b_1, \ldots, b_r, b_{r+1}, \ldots, b_n)$), we have 
$$\deg(u \otimes v)=(a_1+b_1, \ldots, a_r+b_r, b_{r+1}, \ldots, b_n).$$

\begin{rem}\label{E_x rem}
(1) If $\rank  x \le 1$, the structure of $\EE_x$ is rather simple. 
In fact $\EE_{\hat{0}} \cong K[t_z^{-1}\mid z \in P^*]$, as we have mentioned above. Similarly, we have $\EE_{y_i} \cong K[t_i^{\pm 1}]\otimes_K K[t_z^{-1} |  z \in P^* \setminus \{ y_i \} ]$. 
These are special cases of the injective modules treated in  \cite[Proposition~3.1.3 (2)]{GW}. 

(2) $S=K[t_x \mid x \in P^* ]$ is the coordinate ring of an additive group $K^{\# P^*}$. The above ring homomorphism  
$\Delta: S \to S \otimes_K S$ is the {\it comultiplication map} of $S$ as a Hopf algebra. We have no idea why the comultiplication map appears in our context.   

(3) For $z_1, \ldots z_m \in P^*$ (we allow the case $z_i =z_j$ for some $i \ne j$), we have 
$$\Delta(t_{z_1}t_{z_2} \cdots t_{z_m})=\sum_{T \subset [m]} \left(\prod_{i \in T} t_{z_i} \otimes \prod_{j \in [m] \setminus T} t_{z_j} \right),$$
where $[m]:=\{1, \ldots, m\}$. 

(4) With the above situation (especially, $x\in [y_1 \vee \cdots \vee y_r]$), for $\ba=(a_1, \ldots, a_n) \in \ZZ^n$, 
$[\EE_x]_\ba \ne 0$ if and only if $a_{r+1}=\cdots =a_n \le 0$. Hence $\EE_x \cong \EE_x(\ba)$ if and only if $a_{r+1} = \cdots = a_n = 0$.
If $\rank x \ge 2$, then $\dim_K [\EE_x]_\ba =\infty$ often occurs. 
\end{rem}

\begin{exmp}
For the simplicial poset $P_1$ of Example~\ref{running}, we have  
$$\EE_x= K[t_1^{\pm 1},t_2^{\pm 1}] \otimes_K K[t_x^{-1}, t_z^{-1}],$$
and its module structure is given as follows. 
For $t_1^a t_2^b \otimes t_x^{-c}t_z^{-d} \in  \EE_x$ ($a,b \in \ZZ$, $c,d \in \NN$), we have 
$$x \cdot (t_1^a t_2^b \otimes t_x^{-c}t_z^{-d} )= t_1^{a+1}t_2^{b+1} \otimes  t_x^{-c}t_z^{-d}+ t_1^a t_2^b \otimes t_x^{-c+1}t_z^{-d}, $$ 
$$t_z \cdot (t_1^a t_2^b \otimes x^{-c}z^{-d})=t_1^a t_2^b \otimes t_x^{-c}t_z^{-d+1},$$ 
$$t_1 \cdot (t_1^a t_2^b \otimes t_x^{-c}t_z^{-d} )= t_1^{a+1}t_2^b \otimes  t_x^{-c}t_z^{-d},$$
$$t_2 \cdot (t_1^a t_2^b \otimes t_x^{-c}t_z^{-d} )= t_1^at_2^{b+1} \otimes  t_x^{-c}t_z^{-d}.$$
Here, if $c=0$ (resp. $d=0$), then $t_x^{-c+1}=0$ (resp. $t_z^{-d+1}=0$).   
\end{exmp}

To show injectivity and some other properties of $\EE_x$, we use a transformation of variables. 
For each $z\in P^*$, set\\
\begin{equation}\label{wt}
\wt_z:=
\begin{cases}
t_z-\disp\prod_{\substack{1\leq i \leq n\\y_i\leq z}} t_i  & \text{if $z\leq x$ and $\rank z\geq 2$,} \\
t_z & \text{otherwise.}
\end{cases}
\end{equation}
For the monomial $\alpha \in \EE_x$ in \eqref{monomials in E_x}, an easy calculation shows that 
$$\wt_{z'} \cdot \alpha = \prod_{i=1}^r t_i^{a_i} \otimes t_{z'} \prod_{z \in P^*_{-x}} t_z^{-a_z}$$
for all $z'\in P^*_{-x}$. In particular, if $a_{z'}=0$, then $\wt_{z'}\alpha=0$. 

Note that $\wt_z$ is a homogeneous element in the $\ZZ^n$-grading with $\deg t_z=\deg \wt_z$ for all $z \in P^*$, $\wt_i=t_i$ for all $1 \le i \le n$, and $S=K[\wt_z \mid z \in P^*]$.  
Consider the {\it subrings} $\wS_x:= K[\wt_1, \ldots, \wt_r]=K[t_1, \ldots, t_r]$ and $\wS_{-x}:= K[\wt_z \mid z \in P^*_{-x}]$ of $S$. (Recall that $S_x$ and $S_{-x}$ are {\it quotient rings}.) 
Clearly, we have $S_x \cong \wS_x$ and $S_{-x} \cong \wS_{-x}$ as rings.  Moreover,  
\begin{equation}\label{variable transformation}
S = \wS_x\otimes_K \wS_{-x}.  
\end{equation} 
Let $\fm_{-x}:= (\wt_z \mid z \in P^*_{-x})$ be the graded maximal ideal of $\wS_{-x}$. Then we have $\fp_x=\fm_{-x}S$. 
We regard the $\wS_x \otimes_K \wS_{-x}$-module  
$$\wE_x:= \wEA \otimes_K \wESx,$$ 
as an $S$-module via the isomorphism \eqref{variable transformation}.

\begin{lem}
We have $\EE_x \cong \wE_x$ as $\ZZ^n$-graded $S$-modules. 
\end{lem}

\begin{proof}
Since $\wt_z \EA=0$ for all $ z\in P^*_{-x}$ and $\wt_i \ESx=0$ for all $i=1, \ldots,r$, 
$$\EE_x \ni  \prod_{i=1}^r t_i^{a_i} \otimes \prod_{z\in P^*_{-x}} t_z^{-a_z} \longmapsto 
 \prod_{i=1}^r \wt_i^{a_i}  \otimes \prod_{z\in P^*_{-x}} \wt_z^{-a_z} \in \wE_x$$
 gives a $\ZZ^n$-graded isomorphism.  
\end{proof}

So, in the sequel, we freely identify $\EE_x$ with $\wE_x$.  
We also remark that, extending the $S$-module structure and the $K[t_1^{\pm 1}, \ldots, t_r^{\pm 1}]$-module structure of $\EE_x$, we can regard $\EE_x$ as an $S[t_1^{-1},\ldots, t_r^{-1}]$-module in the natural way. 

Clearly, $M \in \*Mod S$ admits the natural $\ZZ$-grading. 
In fact, if $\deg \alpha =(a_1, \ldots, a_n) \in \ZZ^n$ for $\alpha \in M$, then $\deg_\ZZ \alpha=\sum_{i=1}^n a_i$.  
For $m \in \NN$, let $(\EE_x)^{-m}$ be the $K$-subspace of $\EE_x$ spanned by 
$$\{  u \otimes v \in \EE_x \mid \deg_\ZZ v=-m  \},$$
and set 
$$(\EE_x)^{\ge -m} :=\bigoplus_{i\ge -m} (\EE_x)^{-i}.$$ Note that $(\EE_x)^{\ge -m}$ is an $S$-submodule of $\EE_x$. 
We have  $(\EE_x)^{\ge 0} =(\EE_x)^0 \cong \EA$ as an $S$-modules.  
Moreover, $S_x$ can be seen as a submodule of $\EA$ in the canonical way,  and we will regard $\EA$ and $S_x$ as submodules of $\EE_x$ in this way. 

\begin{prop}\label{ess ext}
$\EE_x$ is a *essential extension of $S_x$. 
\end{prop}

\begin{proof}
Since $\EA$ is a *essential extension of $S_x$, it suffices to show  that $\EE_x$ is  a *essential extension of $\EA$. For  $0 \ne \alpha \in \EE_x$, take the minimum $m$ such that $\alpha = \sum_i c_i u_i \otimes v_i \in (\EE_x)^{\ge -m}$, where, for each $i$, $0 \ne c_i \in K$, and $u_i \in \EA$ and $v_i \in \ESx$ are Laurent monomials. 
We will prove the assertion by induction on $m$. There is nothing to prove if $m=0$. So we may assume that $m >0$. Then we can take $z \in P^*_{-x}$ with $t_z v_i \ne 0$ for some $i$. Then we have 
$0 \ne \wt_z \alpha \in (\EE_x)^{\ge -m+1}$. By the induction  hypothesis, there exists $f \in S$ such that $0 \ne (f \wt_z) \cdot \alpha=f \cdot (\wt_z \alpha ) \in \EA$, and we are done. 
\end{proof}

The category $\*Mod S$ admits enough injectives, and an indecomposable injective is isomorphic to a degree shift $\EE_S(S/\fp)(\ba)$ for $\ba \in \ZZ^n$ of the *injective envelope $\EE_S(S/\fp)$ for some $\ZZ^n$-graded primes ideal $\fp$ (c.f. \cite[Theorem~1.3.3]{GW}), and any injective is isomorphic to a direct sum of indecomposable ones. As in the nongraded case, the graded Bass number of $M \in \*Mod S$ can be computed by 
$$\mu^i(\fp, M)=\dim_{\kappa(\fp)}(\gExt_S^i
(S/\fp, M) \otimes \kappa(\fp))$$
for a $\ZZ^n$-graded prime ideal $\fp$, where $\kappa(\fp)$ is the quotient field of $S/\fp$. 

\begin{thm}
The above  $\EE_x$ is the *injective envelope of $S_x \, (=S/\fp_x)$.
\end{thm}

\begin{proof}
Since, in Proposition~\ref{ess ext}, we have already shown that $\EE_x$ is a *essential extension of $S_x$, it suffices to show that $\EE_x$ is injective in $\*Mod S$.
By Proposition~\ref{ess ext}, we have $\Ass_S\EE_x =\{\p_x\}$. 
Hence,  if $\mu^i(\fp, \EE_x)\ne 0$ for a $\ZZ^n$-graded prime ideal $\fp$, then $\fp \supset \fp_x$.  
So, to prove the injectivity in $\*Mod S$, it suffices to show that $\gExt_S^1(S/\fp,\EE_x)=0$ for any $\ZZ^n$-graded prime ideal $\fp$ with $\fp \supset \mathfrak{p}_x$. 
Since $\fp \supset I_P$, we have $\fp =\fp_{x'}$ for some 
$x' \in P$ (actually, $x' \le x$). 
Moreover, for the (monomial) prime ideal $\fq:=\fp\cap \wS_x$ of $\wS_x$, we have $S/\fp \cong \wS_x/\fq \otimes_K \wS_{-x}/\m_{-x}$. 
Using the $\mathrm{K\ddot{u}nneth}$ formula in \cite{STY}, we have
\begin{align*}
& \gExt^1_S(S/\fp, \EE_x)\\
\cong &\bigoplus_{\substack{(p,q)=(1,0),\\ \qquad   (0,1)}} {\gExt}^p_{\wS_x}(\wS_x/\fq, \wEA)\otimes_K {\gExt}^q_{\wS_{-x}}( \wS_{-x}/\m_{-x}, \wESx) \\
=&0,  
\end{align*}
since $\gExt^1_{\wS_x}(\wS_x/\fq, \wEA)={\gExt}^1_{\wS_{-x}}( \wS_{-x}/\m_{-x}, \wESx)=0$. 
\end{proof}

For further study, let us find specific elements of $I_P$. As before, we assume that $x \in [y_1 \vee \cdots \vee y_r]$. For a subset $U \subset \{1, \ldots, r\}$, we simply denote $[\bigvee_{i \in U} y_i]$ by $[U]$. 
The set $[U]$ admits a unique element $z$ with $z \le x$. 
Note that,  for $i \in U$, we have 
$$[U]=\bigsqcup_{w \in [U \setminus \{i\}] } [w \vee y_i].$$

\begin{lem}\label{specific elements}
With  the above notation, we have the following.     
\begin{itemize}
\item[(1)] For $U \subset \{ 1, \ldots, r\}$ with $\#U \ge 2$, we have $f_{U}:=\sum_{z \in [U]} \wt_z \in I_P$. 
\item[(2)] Take distinct $z_1, z_2 \in [U]$ with $z_1 \le x$ (hence $z_2 \not \le x$). Then we have 
$$g_{z_1,z_2}:=\wt_{z_1}\wt_{z_2}+ \wt_{z_2} \cdot  \prod_{i \in U} \wt_i  \in I_P.$$
\end{itemize}
\end{lem}

\begin{proof}
(1) We will prove the assertion by induction on $\# U$. 
First, consider the case $\# U=2$. Then we may assume that $U=\{1, 2\}$ and $[U]=\{ z_1, \ldots, z_k\}$ with $z_1 \le x$.  Then we have 
$$I_P \ni f_{y_1,y_2}=t_1 t_2 -\sum_{i=1}^k t_{z_i}
= t_1 t_2 -t_{z_1}-\sum_{i=2}^k t_{z_i}
=  -\wt_{z_1}-\sum_{i=2}^k \wt_{z_i}=-\sum_{z \in [U]} \wt_z .$$

Next, consider the case $m:=\# U \ge 3$. We may assume that $U=\{1, \ldots, m\}$. Then,  for $z_1 \in [U \setminus \{m\}]$ and $w_1 \in [U]$ with $z_1 < w_1 \le x$, we have $w_1 \in [y_m \vee z_1] \subset [U]$, and hence we can denote $[y_m \vee z_1]=\{w_1, \ldots, w_k \}$. 
In the sequel, $f \equiv g$ for $f,g \in S$ means $f-g \in I_P$. Since $y_m \wedge z_1 =\hat{0}$, we have $t_mt_{z_1} \equiv \sum_{i=1}^k t_{w_i}$.  So 
\begin{eqnarray*}
t_m \wt_{z_1}=t_m (t_{z_1}-t_1\cdots t_{m-1})&=&t_mt_{z_1}-t_1\cdots t_m \\
&\equiv&\sum_{i=1}^k t_{w_i}-t_1\cdots t_m \\
&=& (t_{w_1}-t_1\cdots t_m)+\sum_{i=2}^k t_{w_i}\\
&=& \wt_{w_1}+\sum_{i=2}^k \wt_{w_i}\\
&=&\sum_{i=1}^k \wt_{w_i}=\sum_{w \in [y_m \vee z_1]} \wt_w. 
\end{eqnarray*}
For $z \in [U \setminus \{m\}]$ with $z \ne z_1$, we have $\wt_z =t_z$, $y_m \wedge z =\widehat{0}$, and hence 
$$t_m \wt_z=t_m t_z \equiv \sum_{w \in [y_m \vee z]}t_w=\sum_{w \in [y_m \vee z]}\wt_w.$$ 

By the induction hypothesis, we have  
$\sum_{z \in [U \setminus \{m\}]} \wt_z \in I_P$. Hence 
\begin{eqnarray*}
I_P \ni \, t_m \cdot \sum_{z \in [U \setminus \{m\}]} \wt_z 
&=& t_m \wt_{z_1} +  \sum_{\substack{z \in [U \setminus \{m\}] \\ z \ne z_1}} t_m \wt_z  \\
&\equiv& \sum_{w \in [y_m \vee z_1]} \wt_w+  \sum_{\substack{z \in [U \setminus \{m\}] \\ z \ne z_1}} t_m \wt_z  \\
&\equiv& \sum_{w \in [y_m \vee z_1]} \wt_w+ \sum_{\substack{z \in [U \setminus \{m\}] \\ z \ne z_1}}  \sum_{w \in [y_m \vee z]} \wt_w \\
&=&\sum_{z \in [U \setminus \{m\}]}  \sum_{w \in [y_m \vee z]} \wt_w \\
&=&\sum_{w \in [U]} \wt_w=f_U, 
\end{eqnarray*}
and we get $f_U \in I_P$.

(2) Since $[z_1 \vee z_2]=\emptyset$, we have 
$$I_P \ni f_{z_1,z_2}=t_{z_1}t_{z_2}= \left( \wt_{z_1} +  \prod_{i \in U} \wt_i 
\right)
\wt_{z_2}=g_{z_1,z_2}.$$
\end{proof}

We regard $\ZZ^n$ as a poset by component-wise comparison. In particular, for $\alpha=(a_1, \ldots, a_n) \in \ZZ^n$, $\alpha \ge \b0$ if and only if $a_i \ge 0$ for all $i$. 
For a $\ZZ^n$-graded $S$-module $M=\bigoplus_{\ba \in \ZZ^n} M_\ba$, $M_{\ge \b0}$ denotes the submodule $M=\bigoplus_{\ba \ge \b0} M_\ba$.

\begin{prop}\label{0:I_P}
For any $x\in P$, $$[\gHom_S(A_P, \EE_x)]_{\geq \bf{0}} \cong (0:_{\EE_x} I_P)_{\ge \bf{0}} = S_x.$$ 
\end{prop}

\begin{proof} 
It suffices to show that $(0:_{\EE_x} I_P)_{\ge \bf{0}}=S_x$.  
It is clear that $(0:_{\EE_x} I_P)_{\ge \bf{0}} \supset S_x$, 
and it remains to show the opposite inclusion. 
Take a homogeneous element 
$$0 \ne \alpha=\sum_{i} c_i u_i\otimes v_i \in (0:_{\EE_x} I_P)_{\ge \bf{0}},$$ where $0 \ne c_i \in K$, and $u_i \in \EA$ and $v_i \in \ESx$ are Laurent monomials.  
We will show that $\deg_\ZZ(v_i)=0$ (that is, $v_i=1$) for any $i$ by contradiction. 
Assume that $\deg_\ZZ(v_j) \ne 0$ (equivalently, $\deg_\ZZ(v_j) < 0$), and it is a minimum among all $\deg_\ZZ(v_i)$'s.  Note that $\deg_\ZZ(u_j)$ is a maximum  among all $\deg_\ZZ(u_i)$'s.
Then  $\wt_z v_j \ne 0$ for some $z \in P^*_{-x}$. Since $\deg \alpha \ge \b0$, we have $\deg(t_z) \le \deg (t_x)$ and $\rank z  \ge 2$. Take $U \subset \{1, \ldots, r \}$ such that $z \in [U]$.  We take $f_U \in I_P$ of Lemma~\ref{specific elements}. Since $f_U \cdot \alpha = 0$, there is some $z' \in [U]$ such that $z' \ne z$ and $\wt_{z'}v_k\ne 0$ with $k \ne j$. Take $w \in [U]$ with $w \le x$. At least one of $z$ and $z'$ is different from $w$. If $z \ne w$ (resp. $z'\ne w$), we have $g_{w, z} \cdot \alpha \ne 0$ (resp. $g_{w, z'} \cdot \alpha \ne 0$) under the notation of Lemma~\ref{specific elements}. (To see this, assume that $z \ne w$. Let $u_l \otimes v_l$ be a monomial of maximum $\deg_\ZZ u_l$ among all monomials $u_i \otimes v_i$'s such that $\wt_z v_i \ne 0$.   Then $(\prod_{i \in U} \wt_i) u_l \otimes \wt_z v_l$ survives in $g_{w, z} \cdot \alpha$.)   This is a contradiction, since $g_{w, z}, g_{w,z'} \in I_P$. 
\end{proof}

\section{Homomorphisms between *injective envelopes}
We start this section with the following lemma. 
We simply denote $1 \otimes 1 \in \EA \otimes_K \ESx =\EE_x$ by $1_x$. 

\begin{lem}\label{S_x -> S_z}
For $\alpha \in S_x \subset \EE_x$ and $\varphi \in [\gHom_S(\EE_x, \EE_z)]_{\b0}$ (that is, $\varphi: \EE_x \to \EE_z$ is a morphism in $\*Mod S$), we have $\varphi(\alpha) \in S_z$. In particular, $\varphi(1_x)= c \cdot 1_z$ for some $c \in K$. 
\end{lem}

\begin{proof}
Since $\alpha \in S_x=(0:_{\EE_x} I_P)_{\ge \mathbf 0}$ by Proposition~\ref{0:I_P}, we have $\varphi(\alpha) \in (0:_{\EE_z} I_P)_{\ge \mathbf 0}=S_z$.     
\end{proof}

Let $x,z \in P$.  If $x \not \ge z$, there is no non-zero $S$-homomorphisms $S_x \to S_z$ and $\EE_x \to \EE_z$. 
If $x \ge z$, there is a canonical surjection $\pi: S_x \to S_z$, and any $\ZZ^n$-graded $S$-homomorphism $S_x \to S_z$ coincides with $c \pi$ for some $c \in K$. For these $x, z$, we have $[\gHom_S(\EE_x, \EE_z)]_{\b0} \ne 0$. In fact, the non-zero $S$-homomorphism $S_x \to \EE_z$ which is the composition  $S_x   \twoheadrightarrow S_z \hookrightarrow \EE_z$ can be extended to $\EE_x \to \EE_z$ by the injectivity of $\EE_z$. 
Moreover,  
for $\varphi \in [\gHom_S(\EE_x, \EE_z)]_{\b0}$, we have $\varphi|_{S_x}=c \pi$ for some $c \in K$ by Lemma~\ref{S_x -> S_z}. 

\begin{rem}
If $\rank x \ge 2$, there is $0 \ne \varphi \in [\gHom_S(\EE_x, \EE_z)]_\b0$ with $\varphi|_{S_x}=0$, or equivalently, $\varphi(1_x)=0$.
To see this, set $M:=\EE_x/S_x$.    
Since $\rank x \ge 2$ now, $\alpha:=\pi_x(t_x) \otimes t_x^{-1} \in [\EA \otimes_K \ESx]_0=[\EE_x]_{\b0}$ gives the non-zero element $\overline{\alpha}$ of $M_{\b0}$. 
For all $w \in P^*$, $t_w \alpha$ equals the natural image of $\pi_x(t_wt_x) \otimes t_x^{-1} \in \EE_x$ in $M$. Hence, we have $S \overline{\alpha} \cong S_x$.   
The non-zero $S$-homomorphism $M \supset S \overline{\alpha}  \to \EE_z$ which is the composition 
$S \overline{\alpha} \stackrel{\cong}{\too} S_x   \twoheadrightarrow S_z \hookrightarrow \EE_z$
can be extended to $M \to \EE_z$ by the injectivity of $\EE_z$. Composing this map with the natural surjection $\EE_x \twoheadrightarrow M$, we get  $0 \ne \varphi \in [\gHom_S(\EE_x, \EE_z)]_\b0$ with $\varphi(1_x)=0$. 

For all $\psi \in [\gHom_S(\EE_x, \EE_z)]_\b0$
and the above $\varphi \in [\gHom_S(\EE_x, \EE_z)]_\b0$, we have $(\psi+\varphi)|_{S_x}=\psi|_{S_x}$. 
Hence, the map 
$$[\gHom_S(\EE_x, \EE_z)]_\b0 \ni \psi \longmapsto \psi|_{S_x} \in [\gHom_S(S_x, S_z)]_\b0$$ is far from injective, and we have $\dim_K [\gHom_S(\EE_x, \EE_z)]_\b0=\infty$.  
\end{rem}

\begin{dfn}\label{clean map def}
Recall that, for $m \in \NN$, $\EE_x^{-m}$ is the $K$-subspace of $\EE_x$ spanned by $\{  u \otimes v \in \EA \otimes_K \ESx= \EE_x \mid \deg_\ZZ v=-m  \}.$ 
Set 
$$\Lf(\EE_x):=\left[\bigoplus_{i \ge 1} \EE_x^{-i}\right]_{\ge \bf{0}}$$ (this is just a $K$-subspace of $\EE_x$). 

We say $\psi \in [\gHom_S(\EE_x,\EE_z)]_\b0$ is {\it clean}, if 
$\psi(\Lf(\EE_x))\subset \Lf(\EE_z)$. 
\end{dfn}

Note that $[\EE_x]_{\ge \b0} = S_x \oplus \Lf(\EE_x)$ as $\ZZ^n$-graded $K$-vector spaces (not $S$-modules). 
 Clearly, $\Lf(\EE_x)$ depends on choice of a basis of $\EE_x$, 
and the cleanness of $\psi \in  [\gHom_S(\EE_x,\EE_z)]_{\b0}$ depends on bases of $\EE_x$ of $\EE_z$.  

\begin{lem}\label{1x1}
For a non-zero clean map $\psi: \EE_x \to \EE_z$, we have $\psi(1_x) \ne 0$.     
\end{lem}

\begin{proof}
For contradiction, assume that $\psi(1_x)=0$, or equivalently, $\psi|_{S_x}=0$.  
Since $\EE_z$ is a *essential extension of $S_z$, we have $\Im \psi \cap S_z \ne 0$.  So we have  a homogeneous element $\alpha \in \EE_x$ with $0 \ne \psi(\alpha) \in S_z$. By the present assumption that $\psi|_{S_x}=0$, we may assume that $\alpha \in \Lf(\EE_x)$. It contradicts the cleanness of $\psi$. 
\end{proof}

\begin{lem}\label{degree 0}
A morphism $\psi \in  [\gHom_S(\EE_x,\EE_z)]_{\b0}$ is clean, 
if  $\psi(\alpha)  \in \Lf(\EE_z)$ holds for all $\alpha \in [\Lf(\EE_x)]_{\b0}$. 
\end{lem}

\begin{proof}
 We may assume that $x \in [y_1 \vee \cdots \vee y_r]$ and  $z \in [y_1 \vee \cdots \vee y_l]$ for $r \ge l$. Note that $\EE_x$ is (resp. $\EE_z$) is a  $\ZZ^r$-graded (resp. $\ZZ^l$-graded) free module over $K[t_1^{\pm 1}, \ldots, t_r^{\pm 1}]$ (resp. $K[t_1^{\pm 1}, \ldots, t_l^{\pm 1}]$), and $\varphi:\EE_x \to \EE_z$ is $K[t_1^{\pm 1}, \ldots, t_l^{\pm 1}]$-linear. 

It suffices to show that $\psi(\alpha)  \in \Lf(\EE_z)$ for all homogeneous elements $\alpha \in \Lf(\EE_x)$. 
If the $i$-th component of $\ba:=\deg \alpha \in \ZZ^n$ is positive for some $i >l$, then we have $[\EE_z]_\ba=0$ and hence $\psi(\alpha)=0$. So we may assume that the $i$-th component of $\deg \alpha \in \ZZ^n$ is $0$ for all $i >l$, that is, $\ba=(a_1, \ldots, a_l, 0, \ldots, 0) \in \ZZ^n$. Then $\alpha':=\left( \prod_{i=1}^l t_i^{-a_i}\right) \cdot \alpha \in  [\Lf(\EE_x)]_{\b0}$. 
By the assumption, we have $\psi(\alpha') \in \Lf( \EE_z)$. Hence $$\psi(\alpha)=\left( \prod_{i=1}^l t_i^{a_i} \right) \cdot \psi(\alpha') \in \Lf(\EE_z).$$ 
\end{proof}

\begin{rem}\label{rank=1}
Since $[\EE_{y_i}]_{\b0}=[S_{y_i}]_{\b0}=K$, $\varphi \in [\gHom_S(\EE_{y_i}, \EE_{\hat{0}})]_{\b0}$ is always clean by the above lemma. 
\end{rem}

\begin{lem}\label{restriction to S_x}
If $\varphi, \psi \in [\gHom_S(\EE_x,\EE_z)]_{\b0}$ are clean and $\varphi |_{S_x}= \psi|_{S_x}$ (equivalently,  $\varphi(1_x)= \psi(1_x)$), we have $\varphi = \psi$.  
\end{lem}

\begin{proof}
Note that $f :=\varphi -\psi \in  [\gHom_S(\EE_x,\EE_z)]_{\b0}$ is clean,  and $f(1_x)=0$. By Lemma~\ref{1x1}, we have $f=0$ and hence $\varphi = \psi$. 
\end{proof}

\begin{lem}\label{clean basic}
For  $\varphi \in [\gHom_S(\EE_x,\EE_z)]_\b0$ and  $\psi \in [\gHom_S(\EE_z,\EE_w)]_\b0$, we have the following, 
\begin{itemize}
\item[(1)] If $\varphi$ and $\psi$ are clean, the composition $\psi \circ \varphi$ is clean.  
\item[(2)] If $\psi$ and $\psi \circ \varphi$ are clean and $\psi \ne 0$, then $\varphi$ is clean.
\end{itemize}

\end{lem}

\begin{proof}
(1) Clear. 

(2) For contradiction, assume that $\varphi$ is not clean. By Lemma~\ref{degree 0}, there is some  $\alpha \in  [\Lf(\EE_x)]_{\mathbf 0}$ with  $\varphi(\alpha) \not \in \Lf(\EE_z)$. 
Take $0 \ne \beta = c\cdot 1_z \in S_z$ and $\gamma \in \Lf(\EE_z)$ with $\psi(\alpha)=\beta+\gamma$.  We have $0 \ne \psi(\beta) \in S_w$ by Lemmas~\ref{S_x -> S_z} and \ref{1x1}, 
and $\psi(\gamma) \in \Lf(\EE_w)$ by the cleanness of $\psi$. Hence 
 $$\psi \circ \varphi(\alpha)=\psi(\varphi(\alpha))=\psi(\beta+\gamma)=\psi(\beta)+\psi(\gamma) \not \in \Lf(\EE_w),$$
and it contradicts the assumption that  $\psi \circ \varphi$ is clean. 
\end{proof}

By Remark~\ref{rank=1} and Lemma~\ref{restriction to S_x}, 
 a $\ZZ^n$-graded $S$-homomorphism $\EE_{y_i} \to \EE_{\widehat{0}}$ is unique up to a scalar multiplication. 
For $y =y_i$, $$\{ t_y^a \otimes \prod_{z \in P^*_{-y}}t_z^{-b_z} \mid a \in \ZZ, b_z \in \NN \, (z \in P^*_{-y})\, \}$$ forms a basis of $\EE_y$, and  
$\psi \in [\gHom_S(\EE_y, \EE_{\widehat{0}})]_{\b0}$ with $\psi(1_y)=1_{\widehat{0}}$ satisfies 
$$\psi( t_y^a \otimes \prod_{z \in P^*_{-y}}t_z^{-b_z})= 1 \otimes \prod_{z \in P^*_{-y}}t_z^{-b_z}t_y^a,$$
where the right hand side is 0 if $a >0$.  

For $x,x' \in P^*$ such that $x$ covers $x'$ (hence $\rank_P x \ge 2$), we will construct a clean $\ZZ^n$-graded $S$-homomorphism $\varphi: \EE_x \to \EE_{x'}$. We may assume that $x \in [y_1 \vee \cdots \vee y_r]$ and $x' \in [y_1 \vee \cdots \vee y_{r-1}]$.  
Set 
\begin{equation}\label{def of Z}
Z:= \{ z \in P \mid z \le x, z \not \le x', z \ne y_r \}
\end{equation}
and 
\begin{equation}\label{def of W}
W:= P^* \setminus (Z \cup \{y_1, \ldots, y_r\}).
\end{equation}
Note that $$P^*=  \{y_1, \ldots, y_r\} \sqcup Z \sqcup W.$$
For $z \in P^*$ with $z \le x$, $z \in Z$ if and only if $z > y_r$. 

Consider a Laurent monomial $\prod_{i=1}^r t_i^{a_i} \in K[t_1^{\pm 1}, \ldots, t_r^{\pm 1}]$, $\bt_Z^{-\bb} :=\prod_{z \in Z} t_z^{-b_z}$ for $\bb =(b_z)_{z \in Z} \in \NN^Z$ and $\bt_W^{-\bc} :=\prod_{w \in W} t_w^{-c_w}$ for $\bc=(c_w)_{w \in W} \in \NN^W$. Then the set of monomials of the form 
\begin{equation}\label{alpha}
\alpha= \left (\prod_{i=1}^r t_i^{a_i} \right )
\otimes \bt_Z^{-\bb} \bt_W^{-\bc}
\end{equation}
forms a $K$-basis of $\EE_x$. 
For a monomial $\bt_Z^{\bd} \in S$ in the variables in $Z$, we have $(e_1, \ldots, e_r) \in \ZZ^r$ such that 
$$
\pi_x(\bt_Z^{\bd})\cdot \prod_{i=1}^r t_i^{a_i}  =\prod_{i=1}^rt_i^{e_i}.$$
Then we set 
$$(\bt^{\bd}_Z | \alpha):=
\displaystyle \prod_{z \in Z} \binom{b_z+d_z}{b_z} \cdot\left( \prod_{i=1}^{r-1}t_i^{e_i}  \otimes \bt_Z^{-\bb-\bd} \cdot \bt_W^{-\bc} \cdot t_r^{e_r} \right) \in \EE_{x'}.
$$
Here we interpret $(\bt^{\bd}_Z | \alpha)=0$ if $e_r>0$. 
Note that the monomial $(\bt^{\bd}_Z | \alpha)$ has the same degree as $\alpha$. 
Now we define $\psi: \EE_x \to \EE_{x'}$ by 
\begin{equation}\label{def of varphi}
\varphi(\alpha):=\sum_{\bd\in \NN^Z} (\bt^{\bd}_Z | \alpha).
\end{equation}
If $\sum_{z \in Z}d_z > a_r$, then $(\bt^{\bd}_Z | \alpha)=0$. 
Hence, the right side of \eqref{def of varphi} is a finite sum.  
\begin{exmp}
For the simplicial poset $P_1$ of Example~\ref{running}, the *injective envelopes $\EE_x$ and $\EE_{y_1}$ have bases 
$\{t_1^{a_1}t_2^{a_2} \otimes t_x^{-b} t_z^{-c} \mid a_1, a_2 \in \ZZ, b,c \in \NN \}$ and $\{t_1^{a_1} \otimes t_x^{-b} t_z^{-c} t_2^{a_2} \mid a_1 \in \ZZ, -a_2, b,c \in \NN \}$ respectively. With the above notation, for $\EE_x$, we have $Z=\{x\}$ and  $W=\{z\}$, and hence 
$$\psi(t_1^{a_1}t_2^{a_2} \otimes t_x^{-b} t_z^{-c})=\sum_{0 \le d \le -a_2}\binom{b+d}{b} t_1^{a_1+d} \otimes  t_x^{-b-d}t_z^{-c}t_2^{a_2+d}.$$  
\end{exmp}

\begin{prop}\label{description of clean maps}
The above $\psi$ is a $\ZZ^n$-graded $S$-homomorphism. Hence $\psi:\EE_x \to \EE_{x'}$ is a clean map.     
\end{prop}

\begin{proof}
The latter assertion is immediate from the former. We will therefore prove the former. 
With the above notation, for $w \in W$, the element 
$$\wt_w= \begin{cases}t_w - \prod_{y_i \le w} t_i & \text{if $w \le x$ (equivalently, $w \le x'$),}\\
t_w & \text{otherwise,}
\end{cases}
$$
defined in \eqref{wt} works well for both $\EE_x$ and $\EE_{x'}$. 
So we have $\psi(\wt_w\alpha)=\wt_w\psi(\alpha)$. Similarly, we have 
$\psi(t_i\alpha)=t_i \psi(\alpha)$ for all $1 \le i \le r$ (the case $i=r$ is somewhat exceptional, but still easy). 
It remains to show that $\psi(t_{z_0}\alpha)=t_{z_0}\psi(\alpha)$ for all $z_o \in Z$ and all monomial $\alpha\in \EE_x$ in \eqref{alpha}. 
We have $t_{z_0}\alpha= \alpha'+\alpha''$, where 
$$\alpha'= \pi_x(t_{z_0}) \cdot \prod_{i=1}^r t_i^{a_i} \otimes \bt_Z^{-\bb} \bt_W^{-\bc} \quad \text{and} \quad 
\alpha''= \prod_{i=1}^r t_i^{a_i} 
\otimes t_{z_0}  \bt_Z^{-\bb} \bt_W^{-\bc}.$$
For simplicity, we set 
$$u\otimes v=  \prod_{i=1}^{r-1}t_i^{e_i} \otimes \bt_Z^{-\bb-\bd} \cdot \bt_W^{-\bc} \cdot t_r^{e_r}  \in \EE_{x'}.$$
Then we have 
$$(\bt^{\bd}_Z | \alpha)=C \cdot \binom{b_{z_0}+d_{z_0}}{b_{z_0}} \cdot (u \otimes v)$$
and 
$$t_{z_0}(\bt^{\bd}_Z | \alpha)=C \cdot \binom{b_{z_0}+d_{z_0}}{b_{z_0}} \cdot (u \otimes t_{z_0} v),$$
where
$$C:=\prod_{z_0 \ne z \in Z} \binom{b_z+d_z}{b_z}.$$ 
If $d_{z_0} >0$, set $\bd':= \bd-\be_{z_0} \in \NN^Z$,  
then we have $$(\bt_Z^{\bd'}|\alpha')=C\cdot
\binom{b_{z_0}+d_{z_0}-1}{b_{z_0}} \cdot (u \otimes t_{z_0} v).$$
Similarly, we have 
$$(\bt_Z^\bd | \alpha'')=C \cdot \binom{b_{z_0}-1+d_{z_0}}{b_{z_0}-1} \cdot (u \otimes t_{z_0} v).$$
Hence, if $d_{z_0} >0$, we have 
$$t_{z_0}(\bt^{\bd}_Z \alpha)=(\bt_Z^{\bd'}|\alpha')+(\bt_Z^\bd | \alpha'').$$
If $d_{z_0}=0$, then we have  
$$t_{z_0} \cdot (\bt_Z^\bd | \alpha)=(\bt_Z^\bd|\alpha'')=C \cdot (u \otimes t_{z_0} v)$$
(if $b_{z_0}$  is also 0, the expressions in the above equation are all 0.) Hence, we have $\psi(t_{z_0}\alpha)=t_{z_0}\psi(\alpha)$. 
\end{proof}

\begin{thm}\label{uniqueness}
For $x, z \in P$ with $x \ge z$, there is a clean map $\EE_x \to \EE_z$, which is unique up to a scalar multiplication.     \end{thm}

\begin{proof}
If $x=z$, then the identity map of $\EE_x$ is clean. 
If $x>z$, we take a sequence $x_0 =x, x_1, x_2, \cdots, x_l=z \in P$ such that $x_i$ covers $x_{i+1}$ for each $0 \le i < l$. By Proposition~\ref{description of clean maps}, we have a clean map $f_i :\EE_{x_i} \to \EE_{x_{i+1}}$. Then the  composition $f_{l-1} \circ \cdots \circ f_1 \circ f_0$ gives a clean map $\EE_x \to \EE_z$ by Lemma~\ref{clean basic} (1). 
The uniqueness (up to a scalar multiplication) follows from Lemma~\ref{restriction to S_x}.  
\end{proof}

For an element $\omega= \sum c_\beta \beta \in \EE_x$ ($\beta$ is a monomial,  $c_\beta \in K$), set $[\omega : \beta]= c_\beta$.

\begin{lem}\label{t_m^-a} 
For a non-zero clean map  $\psi:\EE_x\to \EE_{z}$ and a monomial $\alpha \in \EE_x$, 
the following are equivalent. 
\begin{itemize}
\item[(1)] $[\psi(\alpha): 1_z] \ne 0$, 
\item[(2)] $\alpha=  1_x$.  
\end{itemize}
\end{lem}

\begin{proof}
(2) $\Rightarrow$ (1): Follows from Lemmas~\ref{S_x -> S_z} and \ref{1x1}. 
(1) $\Rightarrow$ (2): Since $\psi$ is $\ZZ^n$-graded, $[\psi(\alpha): 1_z] \ne 0$ implies $\deg \alpha=\b0$. So, if $\alpha \ne 1_x$, then $\alpha \in \Lf(\EE_x)$. Since $\psi$ is clean, we have $\psi(\alpha) \in \Lf(\EE_z)$ and   $[\psi(\alpha): 1_z] = 0$.  
\end{proof}

\begin{thm}\label{clean base}
For $x, z \in P$ with $x \ge z$, and 
$\varphi \in [\gHom_S(\EE_x, \EE_z)]_\b0$ with $\varphi(1_x) \ne 0$, there is a $\ZZ^n$-graded automorphism $\sigma$ of $\EE_x$ such that $\varphi \circ \sigma$ is a clean map.  
\end{thm}

Any $\ZZ^n$-graded automorphism $\sigma$ of $\EE_x$ satisfies 
$\sigma(1_x)= c \cdot 1_x$ for some $0 \ne  c \in K$. Hence, if $\varphi(1_x) = 0$ for $0 \ne \varphi \in [\gHom_S(\EE_x, \EE_z)]_\b0$,  $\varphi \circ \sigma$ cannot be clean for any automorphism $\sigma$ by Lemma~\ref{1x1}.

\begin{proof}
Let $\psi: \EE_x \to \EE_z$ be the clean map with $\psi(1_x)=1_z$. 
It suffices to show that there exists  a $\ZZ^n$-graded automorphism  $\tau$ of $\EE_x$ with $\varphi=\psi \circ \tau$. In fact, $\sigma =\tau^{-1}$ satisfies the desired property.  Since $\tau$ is an automorphism on $E_x$, $\wt_z \in S$ for $z \in P^*$ defined in \eqref{wt} is useful.    

As before, we may assume that $x \in[y_1 \vee \cdots \vee y_r]$ and $z \in[y_1 \vee \cdots \vee y_l]$ with $r \ge l$. Recall that the set of monomials of the form 
$$\beta =\prod_{i=1}^r t_i^{b_i} \otimes \prod_{z \in P^*_{-x}} t_z^{-b_z}$$
for $(b_1, \ldots, b_r) \in \ZZ^r$ and $(b_z)_{z \in P^*_{-x}} \in \NN^{P^*_{-x}}$ forms a $K$-basis of $\EE_x$. Set 
$$f_\beta:= \left(\prod_{i=1}^r t_i^{-b_i}\right)\cdot \left( \prod_{z \in P^*_{-x}} \wt_z^{b_z}\right) \in S[t_1^{-1}, \ldots, t_r^{-1}].$$

For monomials $\alpha, \beta \in \EE_x$, we construct a map $\tau : \EE_x \to \EE_x$ by 
$$[\tau(\alpha):\beta]:=[\varphi(f_\beta \cdot \alpha): 1_z].$$
Note that, in the $\ZZ^n$-grading, $\deg \alpha =\deg \beta$ if and only if $\deg (f_\beta \cdot \alpha) ={\bf 0}$. 
Hence $[\tau(\alpha):\beta] \ne 0$ implies $\deg \alpha =\deg \beta$, and  $\tau$ is $\ZZ^n$-graded. 
For all monomials $\alpha, \beta \in \EE_x$ and all $w \in P^*_{-x}$, we have 
$$
[\tau(\wt_w \alpha):\beta]
=[\varphi(f_\beta \cdot \wt_w\alpha): 1_z] 
=[\varphi(f_{\beta'} \cdot \alpha): 1_z] 
=[\tau(\alpha):\beta'],$$
where 
$$\beta' =\prod_{i=1}^r t_i^{b_i} \otimes t_w^{-1} \prod_{z \in P^*_{-x}} t_z^{-b'_z}.$$
Since $\beta' $ is only monomial with $\beta=\wt_w \beta'$, we have $\tau(\wt_w \alpha)=\wt_w \tau(\alpha)$. 
We can show that $\tau(t_i \alpha)=t_i \tau(\alpha)$ for $1 \le i \le r$ in a similar way, while we set $\beta'= \wt_i^{-1} \beta$ this time.
Summing up, we have shown that $\tau$ is $S$-linear. 

Next we will show that $\tau$ is injective. For contradiction, assume that $\Ker \tau \ne 0$.  By Proposition~\ref{ess ext}, we have $\Ker \tau \cap S_x \ne 0$. However, since $\tau(1_x) \ne 0$ by the construction, $\tau|_{S_x}$ is injetive. This is a contradiction, and $\tau$ is an injection. Since $\tau$ is $\ZZ^n$-graded and $\EE_x$ is *injective, $\tau$ splits and $\tau(\EE_x)$ is a direct summand of $\EE_x$.   Since $\EE_x$ is indecomposable, $\tau$ is isomorphism.  

Finally, we will show that $\varphi=\psi\circ \tau$.  We have 
\begin{equation}\label{1 otimes 1}
[\psi \circ \tau(\alpha): 1_z] = [\tau(\alpha): 1_x] =[\varphi(\alpha): 1_z]
\end{equation}
for all monomials $\alpha \in \EE_x$, where the first (resp. the second) equality follows from  Lemma~\ref{t_m^-a} (resp. the construction of $\tau$). 
Set $\varphi' :=  \varphi-\psi\tau  \in [\gHom_S(\EE_x, \EE_z)]_\b0$. By \eqref{1 otimes 1}, we have $[\varphi'(\alpha) : 1_z] = 0$ for all $\alpha \in \EE_x$. Assume that $\varphi' \ne 0$. 
Then $\Im \varphi' \cap S_z \ne 0$. Since $\varphi'$ is $K[t_1^{\pm 1}, \ldots, t_l^{\pm 1}]$-linear, there is a monomial $\alpha \in \EE_x$ such that $[\varphi'(\alpha) : 1_z] \ne 0$. This is a contradiction. 
\end{proof}

\section{Graded dualizing complexes}
For the foundational theory and basic properties of dualizing complexes, we refer the reader to \cite{Ye}.

Set $N:=\dim S =\# P^*$. 
For a prime ideal $\fp \subset S$, let $E(S/\fp)$ be the injective envelope of $S/\fp$. 
We have the cochain complex $D_S^\bullet$ with 
$$0 \too D_S^{-N}\too D_S^{-N+1}\too \cdots \too D_S^{-1}\too D_S^0\too 0,$$ 
$$D_S^{-i}= \bigoplus_{\substack{\fp \in \Spec S \\ \hght \fp =i}} E_S(S/\fp)$$
satisfying $H^{-N}(D_S^\bullet) \cong S$ and $H^i(D_S^\bullet) =0$ for all $i \ne -N$. That is, $D_S^\bullet$ is a translation of a minimal injective resolution of $S$, and 
quasi-isomorphic to the {\it normalized dualizing complex} of $S$.  

Set $\bw:=\sum_{x \in P^*} \deg t_x \in \NN^n$. Using the Koszul resolution of $K=S/\fm$ over $S$, we see that $\gExt_S^N(K,S(-\bw))\cong K$. Let $\DB_S$  be the minimal injective resolution of $S(-\bw)$ in $\*Mod S$ with the same translation as $D_S^\bullet$, that is,  $H^{-N}(\DB_S) \cong S(-\bw)$ and $H^i(\DB_S) =0$ for all $i \ne -N$ (so if we forget the grading, $\DB_S$ is quasi-isomorphic to $D_S^\bullet$). 
Then $\DB_S$ is quasi-isomorphic to the normalized $\ZZ^n$-graded dualizing complex of $S$, and of the form 
$$0 \too \gD_S^{-N}\too \gD_S^{-N+1}\too \cdots \too \gD_S^{-1}\too \gD_S^0\too 0,$$ 
\begin{equation}\label{gD_S}
\gD_S^{-i}= \bigoplus_{\substack{\fp:\text{$\ZZ^n$-graded prime} \\ \hght \fp =i}} (\EE_S(S/\fp))(-\ba_\fp)
\end{equation}
for some $\ba_\fp \in \ZZ^n$. Since $\gExt_S^N(K,S(-\bw))\cong K$, we have $\ba_\fm=\b0$, that is, $\gD_S^0=\EE_{\widehat{0}}$. This is the reason we use $S(-\bw)$ instead of $S$ itself. 
It is a natural question  whether we can take $\ba_\fp=\b0$ for all $\ZZ^n$-graded prime $\fp$. Proposition~\ref{no shift} bellow is a partial answer.  


This is the graded version of \cite[Theorem 3.9]{S}  for local rings. The proof in the graded setting is essentially the same.

\begin{lem}\label{Hom(A_P,-)}
The cochain complex $\gHom_S(A_P, \DB_S)$ is quasi-isomorphic to the normalized $\ZZ^n$-graded dualizing complex of $A_P$.     
\end{lem}%

The $\ZZ^n$-graded dualizing complex of $A_P$ is described in \cite{Y11}. 
For $x \in P$, let $\bfp_x$ denote the  prime ideal of $A_P=S/I_P$ which is the image of $\fp_x \supset I_P$. Any $\ZZ^n$-graded prime ideal of $A_P$ coincides with $\bfp_x$ for some $x \in P$. Let $\EE_{A_p}(S_x)$ be the injective envelope of $S/\bfp_x \cong S_x$ in the category of $\ZZ^n$-graded $A_P$-modules. Note that $(0:_{\EE_x} I_P)\cong \EE_{A_p}(S_x)$, and hence $[\EE_{A_p}(S_x)]_{\ge \b0} \cong S_x$. 
We regard $S_x$ as a submodule of $\EE_{A_p}(S_x)$ in this way. Note that $1_x \in \EE_{A_P}(S_x)$. 

Set $d:=\dim A_P$. As shown in \cite[p.2234]{Y11}, 
the normalized $\ZZ^n$-graded dualizing complex of $A_P$ is quasi-isomorphic to the following complex:
$$\DB_{A_P}: 0 \too \gD_{A_P}^{-d} \too \gD_{A_P}^{-d+1} \too \cdots \too \gD_{A_P}^{-1} \too \gD_{A_P}^0 \too 0$$ 
$$\gD^{-i}_{A_P}=\bigoplus_{\substack{x \in P \\ \rank x=i}} \EE_{A_p}(S_x)$$
with the differential map $\partial^{-i}: \gD_{A_p}^{-i} \to  \gD_{A_p}^{-i+1}$ given by 
\begin{equation}\label{differential}
\gD_{A_p}^{-i} \supset \EE_{A_P}(S_x) \ni 1_x \longmapsto \sum_{\substack{x' \in P\\ \text{$x$ covers $x'$}}}\pm 1_{x'} \in\bigoplus_{\substack{z \in P \\ \rank z=i-1}} \gE_{A_p}(S_z) = \gD_{A_p}^{-i+1}.
\end{equation}
Here the sign $\pm$ is given by an incidence function associated with the regular CW complex associated with $P$. 

\begin{rem}
As shown in \cite[Theorem~1.1]{Y11}, the subcomplex $I^\bullet_{A_P}:= [\gD_{A_p}^\bullet]_{\ge \b0}$ is quasi-isomorphic to $\gD_{A_p}^\bullet$ itself.   Clearly,   $I^\bullet_{A_P}$ is of the form 
$$0 \too I_{A_P}^{-d} \too I_{A_P}^{-d+1} \too \cdots \too I_{A_P}^{-1} \too I_{A_P}^0 \too 0$$ with
$$I^{-i}_{A_P}=\bigoplus_{\substack{x \in P \\ \rank x=i}} S_x.$$
\end{rem}

\begin{prop}\label{no shift}
In the situation of \eqref{gD_S}, we can take $\ba_{\fp_x}=\b0$ for all $x \in P$.      
\end{prop}

\begin{proof}
The assertion follows Lemma~\ref{Hom(A_P,-)} and the above description of $\DB_{A_P}$. 
\end{proof}

Let $\Gamma_{I_P}: \*Mod S \to \*Mod S$ be the local cohomology functor with supports in $I_P$, that is,
$$
\Gamma_{I_P}(\EE(S/\fp))=
\begin{cases}
\EE(S/\fp) & \text{$\fp \supset I_P$, equivalently, $\fp=\fp_x$ for $\exists x \in P$},\\
0 & \text{otherwise} 
\end{cases}
$$
for a $\ZZ^n$-graded prime ideal $\fp$. 
The $(i-N)$-th cohomology $H^{i-N}(\Gamma^\bullet)$ of the cochain complex $\Gamma^\bullet :=\Gamma_{I_P}(\DB_S)$  
is the $i$-th local cohomology $H_{I_P}^i(S(-\bw)) \, (\cong H_{I_P}^i(S)(-\bw))$.  
By Proposition~\ref{no shift}, $\Gamma^\bullet$ is of the form 
$$0 \too \Gamma^{-d}\too \Gamma^{-d+1}\too \cdots \too \Gamma^{-1}\too \Gamma^0\too 0$$ 
with 
$$\Gamma^{-i}=\bigoplus_{\substack{x \in P \\\rank x=i}} \EE_x.$$

\begin{dfn}\label{clean map def2}
For a module $E=\bigoplus_{x \in P} \EE_x^{l_x} \in \*Mod S$ with $l_x \in \NN$ (we fix a basis for each direct summand $\EE_x$), set 
$$\Lf(E):=\bigoplus_{x \in P} \Lf(\EE_x)^{l_x}.$$

Take $E=\bigoplus_{x \in P} \EE_x^{l_x}$ and $F =\bigoplus_{x \in P} \EE_x^{m_x}$ with $l_x, m_x \in \NN$. We say $\psi \in [\gHom_S(E,F)]_{\b0}$ is {\it clean}, if $\psi(\Lf(E)) \subset \Lf(F)$.  
\end{dfn}

\begin{thm}\label{diff are clean}
We can choose bases of $\EE_x$ for all $x \in P$ so that the differential maps $\partial^{-i}: \Gamma^{-i}\to \Gamma^{-i+1}$ 
of  $\Gamma_{I_P}(\DB_S)$ is clean for all $i$. 
\end{thm}

\begin{proof}
Throughout this proof, $x,z \in P$ are elements such that  $\rank x=i$ and $x$ covers $z$. Let $\iota_x: \EE_x \to \Gamma^{-i}$ be  the injection, and  $\pi_z: \Gamma^{-i+1} \to \EE_z$ the projection of the direct sums. 
(The symbols $\iota_-$ and $\pi_-$ will be used in this way.) 
It suffices to show that we can choose bases so that 
$\pi_z \partial^{-i}\iota_x: E_x \to E_z$ is clean. 
By Lemma~\ref{Hom(A_P,-)} and \eqref{differential}, we have $\pi_z \partial^{-i}\iota_x(1_x)=\pm 1_z$.

We construct the basis of $\EE_x$ inductively on $i=\rank x$. 
If $\rank x=0$ (equivalently, $x=\widehat{0}$), there is nothing to do. 
If $\rank x=1$ (and hence $z=\widehat{0}$), $\pi_z \partial^{-1} \iota_x$ is automatically clean by Remark~\ref{rank=1}. Suppose that we have chosen the bases of $\Gamma^{-i+1}, \Gamma^{-i+2}, \ldots$ satisfying the desired property. Take $x \in P$ with $\rank x=i$. If $z \in P$ is covered by $x$, we can choose the basis $\EE_x$ so that $\pi_z \partial^{-i}\iota_x$ is clean by Theorem~\ref{clean base}. Since $P$ is simplicial, if $z' \in P$ is  also covered by $x$, then there is $w \in P$ covered by both $z$ and $z'$. 
Moreover, if $z'' \in P$ satisfies $w < z'' <x$, then either $z''=z$ or $z''=z'$ holds. 
Since $\partial^{-i+1}\partial^{-i}=0$, we have $$(\pi_w\partial^{-i+1}\iota_z)\circ (\pi_z\partial^{-i}\iota_x)+(\pi_w\partial^{-i+1}\iota_{z'})\circ (\pi_{z'}\partial^{-i}\iota_x)=0.$$  

By the assumption, $(\pi_w\partial^{-i+1}\iota_z)$, $(\pi_w\partial^{-i+1}\iota_{z'})$, and  $(\pi_z\partial^{-i}\iota_x)$ are clean. By Lemma~\ref{clean basic} (1), the composition $(\pi_w\partial^{-i+1}\iota_z) \circ (\pi_z\partial^{-i}\iota_x)$ is clean. 
Hence $(\pi_w\partial^{-i+1}\iota_{z'})\circ (\pi_{z'}\partial^{-i}\iota_x)$, which equals $-(\pi_w\partial^{-i+1}\iota_z)\circ (\pi_z\partial^{-i}\iota_x)$, is also. 
Since $\pi_w\partial^{-i+1}\iota_{z'} \ne0$, 
$\pi_{z'}\partial^{-i}\iota_x$ is clean by Lemma~\ref{clean basic} (2). Choosing bases of $\EE_{x'}$ for all $x' \in P$ with $\rank x'=i$ in this way, we have that $\partial^{-i}$ is clean (with respect to this bases). 
\end{proof}

\section*{Acknowledgments}
 We are deeply grateful to Professor Mitsuyasu Hashimoto for his valuable comments. We appreciate the early discussions with Professor Luis Núñez-Betancourt. 
We also thank Dr. Xin Ren  for his helpful support in preparing this manuscript. 
The first author is supported by JSPS KAKENHI Grant Number 22K03258. The second author is supported by JSPS KAKENHI Grant Numbers 22K03258 and 25K06928.

\end{document}